\documentclass[12pt]{amsart}
\usepackage{amsmath,amssymb}
\usepackage{hyperref}
\usepackage{longtable}
\usepackage{xy,amscd}
\usepackage{epsfig}
\usepackage{xspace}
\usepackage{tikz}
\usetikzlibrary{calc}
\xyoption{all}
\usepackage{mathabx}

\newtheorem{propo}{Proposition}[section]
\newtheorem{corol}[propo]{Corollary}
\newtheorem{theor}[propo]{Theorem}

\theoremstyle{definition}
\newtheorem{defin}[propo]{Definition}
\newtheorem{examp}[propo]{Example}

\theoremstyle{remark}
\newtheorem{remar}[propo]{Remark}

\numberwithin{equation}{section}

\newcommand{\NN }{\mathbb{N}}
\newcommand{\CC }{\mathbb{C}}

\newcommand{\ZZ }{\mathbb{Z}}

\newcommand{\id }{\mathrm{id}}

\newcommand{\cc }{{\bf c}}

\newcommand{\cAp }{\mathcal{E}}

\newcommand{\Dc }{\mathcal{D}}
\newcommand{\Fc }{\mathcal{F}}

\newcommand{\quidcyc}{quiddity cycle\xspace}
\newcommand{\quidcycs}{quiddity cycles\xspace}
\newcommand{\qtriple}{\text{\bf q}}

\newcommand{\blp}{\boldsymbol{(}}
\newcommand{\brp}{\boldsymbol{)}}

\newcommand{\dynk}[3]{
$\begin{tikzpicture}
\draw (0,0) -- (1.2,0);
\draw (-0.08,0) circle (0.08);
\draw (1.28,0) circle (0.08);
\coordinate[label=0:$#1$] (q1) at (-0.32,0.36);
\coordinate[label=0:$#2$] (q) at (0.27,0.30);
\coordinate[label=0:$#3$] (q2) at (1.0,0.36);
\end{tikzpicture}$
\hspace{-16pt}
}

\title[On subsequences of quiddity cycles and Nichols algebras]
{On subsequences of quiddity\\ cycles and Nichols algebras}

\author{M.~Cuntz}
\address{Michael Cuntz,
Institut f\"ur Algebra, Zahlentheorie und Diskrete Mathematik,
Fakult\"at f\"ur Mathematik und Physik,
Leibniz Universit\"at Hannover,
Welfengarten 1,
D-30167 Hannover, Germany}
\email{cuntz@math.uni-hannover.de}

\begin{document}

\begin{abstract}
We provide a tool to obtain local descriptions of quiddity cycles.
As an application, we classify rank two affine Nichols algebras of diagonal type.
\end{abstract}
% keywords = quiddity cycle; Nichols algebra; Weyl groupoid; frieze pattern; Catalan number
% subject  = 13F60; 16T20
% 11 pages, 2 figures

\maketitle

\section{Introduction}

Consider any structure counted by Catalan numbers, for example triangulations of a convex polygon by non-intersecting diagonals.
Given such a triangulation (see for example Fig.\ \ref{ex_124221423}), the sequence of numbers of triangles at each vertex has been called its \emph{quiddity cycle} by Conway and Coxeter \cite{jChC73} because this sequence is the first row of any \emph{frieze pattern} and thus contains the ``quiddity'' of the frieze.
The main result of this paper is:
\begin{theor}[Thm.\ \ref{etaseqsubsets}]
For any $\ell\in\NN$ we may compute finite sets of sequences $E$ and $F$, where the elements of $F$ have length at least $\ell$, and such that every quiddity cycle not in $E$ has an element of $F$ as a (consecutive) subsequence.
\end{theor}
\noindent
In other words, this theorem gives a local description of quiddity cycles. For example if $\ell=4$:
\begin{corol}\label{etafour}
Every \quidcyc \footnote{Quiddity cycles are considered up to the action of the dihedral group, for example $\blp 1,2,3,1,2,3\brp\equiv\blp 2,1,3,2,1,3\brp$.} $c\notin \{\blp 0,0\brp,\blp 1,1,1\brp,\blp 1,2,1,2\brp\}$ contains at least one of
\begin{eqnarray*}
&&(1,2,2,1),(1,2,2,2),(1,2,2,3),(1,2,2,4),(1,2,3,1),(1,2,3,2),\\
&&(1,2,3,3),(1,2,4,1),(1,2,4,3),(1,2,5,1),(1,2,5,2),(1,2,6,1),\\
&&(1,3,1,3),(1,3,1,4),(1,3,1,5),(1,3,1,6),(1,3,4,1),(1,4,1,2),\\
&&(1,5,1,2),(1,6,1,2),(1,7,1,2),(2,1,3,2),(2,1,3,3),(2,2,1,4),\\
&&(2,2,1,5),(3,1,2,3),(3,1,2,4).
\end{eqnarray*}
\end{corol}
\noindent
For instance, the cycle in Fig.\ \ref{ex_124221423} contains $(1,2,2,4)$.

\begin{figure}
\centering
\includegraphics[width=0.3\columnwidth]{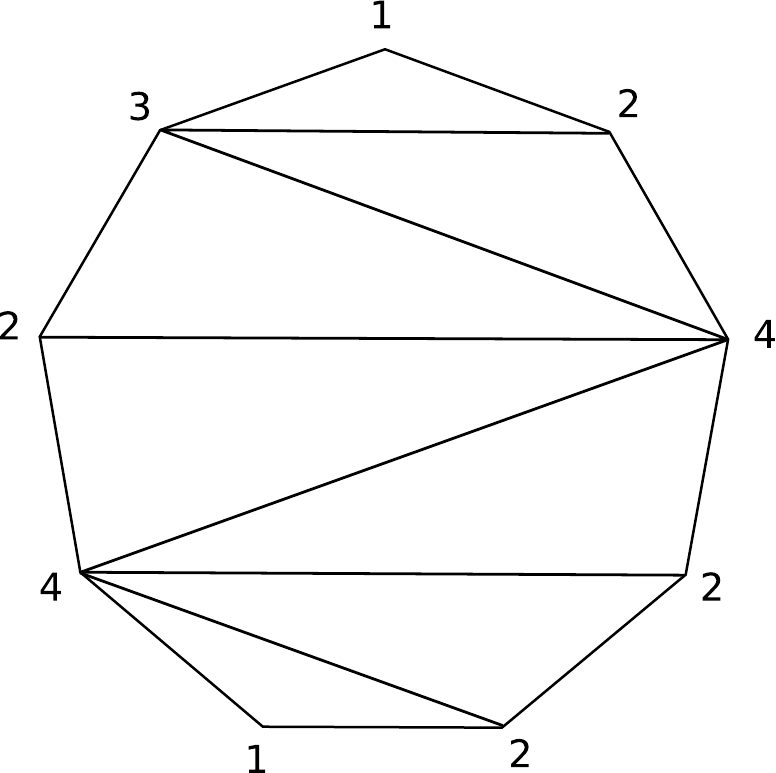}
\caption{Triangulation with quiddity cycle $(1,3,2,4,1,2,2,4,2)$.\label{ex_124221423}}
\end{figure}

As an application, we reproduce the classification of
\cite{jW16}\footnote{Notice that J.\ Wang's result also relies on our main theorem. In this paper we considerably shorten the proof (see Thm.\ \ref{thm:subseqs}) since Wang requires a computer whereas our computation is entirely performed by hand.}:
Nichols algebras of diagonal type are straight forward generalizations of certain algebras investigated by Lusztig in the classical theory. They also produce a rich set of combinatorial invariants called the Weyl groupoids. The Weyl groupoid (see for example \cite{p-H-06},\cite{p-CH09a}) was the essential tool in Heckenberger's classification of finite dimensional Nichols algebras of diagonal type in which case the real roots of the Weyl groupoid form a finite root system \cite{p-Heck06b}.

A natural question is to classify those Nichols algebras for which the Weyl groupoid produces infinitely many roots.
Of course, such a classification is probably a long term project, thus the very first question is to classify only the Nichols algebras of rank two, and only those for which the root system defines a Tits cone which is a half space, i.e., the corresponding Weyl groupoid is then called ``affine''.

The classification of affine Weyl groupoids of rank two was treated in joint work with Bernhard M\"uhlherr and Christian Weigel which started at an Oberwolfach Miniworkshop in October 2012, see \cite{CMW}, \cite{p-CMa-15}, \cite{CMWW}.
It turns out that an affine Weyl groupoid of rank two may be identified with its ``characteristic sequence'', an infinite sequence of positive integers.
Theorem \ref{affwgd:ranktwo} gives an explicit construction of such a characteristic sequence by gluing quiddity cycles together.
These sequences also correspond to certain infinite friezes (see \cite{BPT15}).

It is not too difficult to find out that, if a rank two Nichols algebra produces a Weyl groupoid with root system, then this algebra is uniquely determined by at most three (possibly four in a special case) consecutive entries in the characteristic sequence.
Thus an important result required for the classification of affine rank two Nichols algebras is to reduce the number of potential Weyl groupoids to a finite set of subsequences.
This is where we apply Theorem \ref{etaseqsubsets}:
% A nice trick is Theorem \ref{etaseqsubsets} which produces finite sets of subsequences such that every quiddity cycle contains at least one of these subsequences, i.e.\ this theorem gives a local description of quiddity cycles.
In Corollary \ref{15cases} we conclude that only $15$ cases need to be considered in the case of affine rank two Nichols algebras of diagonal type.
After a short computation, we obtain a complete description of the corresponding bicharacters. Some of them may be identified as classical algebras, others should be investigated more thoroughly.

This paper is organized as follows. The first section is purely combinatorial; we prove the main theorem and a variation required later. In the second section we briefly explain (in very elementary terms) how to obtain the Weyl groupoid from the bicharacter of a rank two Nichols algebra of diagonal type. In the last section, we apply the main result to obtain a classification of affine bicharacters.
% Nichols algebras.

\section{Subsequences of \quidcycs}

\begin{defin}
Let $\Fc_n:=\NN_0^n$ and
\[ \Dc_n := \Fc_n / \sim,  \]
where $(c_1,\ldots,c_n)\sim(d_1,\ldots,d_n)$ if and only if there exists a permutation $\sigma$ in the dihedral group with
$(d_1,\ldots,d_n)=(c_{\sigma(1)},\ldots,c_{\sigma(n)})$, i.e.\ $\sim$ is the equivalence relation obtained by setting
\begin{eqnarray*}
(c_1,\ldots,c_n) &\sim& (c_2,c_3,\ldots,c_{n-1},c_n,c_1),\\
(c_1,\ldots,c_n) &\sim& (c_n,c_{n-1},\ldots,c_2,c_1).
\end{eqnarray*}
We write
\[ \Dc := \bigcup_{n\in\NN} \Dc_n, \quad \Fc := \bigcup_{n\in\NN} \Fc_n. \]
We further write $\blp c_1,\ldots,c_n\brp$ for the equivalence class of $(c_1,\ldots,c_n)$.
\end{defin}
\begin{remar}\label{dih_rem}
Whenever we consider elements of $\Dc$ in this section,
all arguments and definitions are using representatives and are well-defined nevertheless. For instance, speaking of positions $i,i+1$ can be meant to be the last and first position for a representative.
\end{remar}
\begin{defin}
We say that $c=\blp c_1,\ldots, c_n\brp\in\Dc$ \emph{contains} $d=(d_1,\ldots, d_m)\in\Fc$ and write $d\subseteq c$ if there exists a $k\in\NN_0$ and a representative $(c_1,\ldots,c_n)$ of $c$ such that
\[ c_{k+i} = d_{i} \quad \text{for all } i=1,\ldots,m. \]
\end{defin}

\begin{defin}
For $a\in\ZZ$, let
\[ \eta(a) := \begin{pmatrix} a & -1 \\ 1 & 0 \end{pmatrix}. \]
\end{defin}

Notice the rule (compare \cite[Lemma 5.2]{p-CH09b})
\begin{equation}\label{eta_rule}
\eta(a)\eta(b) = \eta(a+1)\eta(1)\eta(b+1)
\end{equation}
for all $a,b$. This motivates the following definition:

\begin{defin}[Compare {\cite[Def.\ 3.2]{p-CH09d}}]\label{eta_seq}
We define the set $\cAp\subseteq\Dc$ of \emph{\quidcycs} recursively as the smallest subset of $\Dc$ satisfying:
\begin{enumerate}
\item $\blp 0,0 \brp\in\cAp$.
\item If $\blp c_1,\ldots,c_n \brp\in\cAp$, then
$\blp c_1+1,1,c_2+1,c_3,\ldots,c_n\brp\in\cAp$. \label{rgca_eta3}
\end{enumerate}
\end{defin}

\begin{remar}\label{etatriangle}
\begin{enumerate}
\item Since $\cAp$ is defined recursively by including ones as in (\ref{eta_rule}), all elements
$\blp c_1,\ldots,c_n\brp\in\cAp$ satisfy the equation $\eta(c_1)\cdots\eta(c_n)=-\id$.
\item A \quidcyc of length $n$ may be visualized by a triangulation of a convex $n$-gon by non-intersecting diagonals:
An entry $c_i$ of the sequence corresponds to the $i$'th vertex; $c_i$ is the number of triangles ending at this vertex.
\end{enumerate}
\end{remar}

\begin{defin}[Compare {\cite{p-C14}}]
Let
\[ \psi : \Fc \rightarrow \Fc, \quad (c_1,c_2,\ldots) \mapsto (c_1+2,1,c_2+2,1,\ldots). \]
Denote
$$\cAp':=\left\{\blp c_1,\ldots,c_n\brp \in\cAp \mid n \text{ is even and } |\{i\mid c_i=1\}|=\frac{n}{2} \right\}.$$
Since two consecutive $1$'s in a $c\in \cAp$ imply $c=\blp 1,1,1\brp$, every $c\in\cAp'$ may be written as
\[ c = \blp *,1,*,1,\ldots,*,1 \brp. \]
For $a\in\ZZ$ let
\[ \xi(a) := \eta(a) \eta(1). \]
Then we have the rule
\[ \xi(a)\xi(3)\xi(b) = \xi(a-1)\xi(b-1) \]
for all $a,b$, and thus obtain that $\psi$ induces a map
\[ \overline{\psi} : \cAp \rightarrow \cAp', \quad \blp c_1,c_2,\ldots\brp \mapsto \blp c_1+2,1,c_2+2,1,\ldots \brp \]
which is a bijection ($\overline{\psi}^{-1}$ corresponds to removing all ears at once in the triangulation).
For $c = (c_1,1,c_3,1,\ldots) \in \Fc_n$, let
\[ \iota (c) := (1,c_1,1,c_3,1,\ldots)\in\Fc_{n+1}. \]
\end{defin}

The key maps to the main result of this section are:

\begin{defin}
For $c=(c_1,\ldots,c_n)\in \Fc$, construct a new sequence $c'$ by repeatedly
replacing
\begin{enumerate}
\item $(\ldots,c_i,c_{i+1},\ldots)$ by $(\ldots,c_i+1,1,c_{i+1}+1,\ldots)$ if $c_i,c_{i+1}>1$,
\item $(c_1,\ldots)$ by $(1,c_1+1,\ldots)$ if $c_1>1$, and
\item $(\ldots,c_n)$ by $(\ldots,c_n+1,1)$ if $c_n>1$,
\end{enumerate}
until none of these rules apply anymore.
We write
\[ \rho : \Fc \rightarrow \Fc, \quad c \mapsto c'. \]
Similarly, define a map $\delta : \Dc\backslash\{\blp 0,0\brp,\blp 1,1,1 \brp\} \rightarrow \cAp'$, $c \mapsto c'$, where $c'$ is obtained from $c$
repeatedly applying the rules
\begin{enumerate}
\item $\blp \ldots,c_i,c_{i+1},\ldots\brp\mapsto\blp\ldots,c_i+1,1,c_{i+1}+1,\ldots\brp$ if $c_i,c_{i+1}>1$,
\item $\blp c_1,\ldots,c_n\brp\mapsto\blp c_1+1,\ldots,c_n+1,1\brp$ if $c_1,c_n>1$.
\end{enumerate}
\end{defin}
\begin{examp}
\begin{eqnarray*}
\rho( (3,1,2,2,1) ) &=& ( 1,4,1,3,1,3,1 ),\\
\rho( (3,1,2,3,1,2) ) &=& ( 1,4,1,3,1,4,1,3,1 ), \\
\delta(\blp 3,1,2,3,1,2\brp ) &=& \blp 4,1,3,1,4,1,3,1\brp, \\
\delta(\blp 4,1,2,2,2,1 \brp ) &=& \blp 4,1,3,1,4,1,3,1 \brp.
\end{eqnarray*}
\end{examp}

\begin{theor}\label{etaseqsubsets}
Assume that $E\subset \cAp$, $F \subset \Fc$ are finite sets such that:
\begin{enumerate}
\item\label{p1} $\{\blp 0,0\brp,\blp 1,1,1 \brp\}\subseteq E$,
\item\label{p2} For all $c\in\cAp$: $\quad c \in E \quad \text{or}\quad \exists f\in F\::\: f\subsetneq c$,
\item\label{p3} Every $f\in F$ contains a $1$.
\end{enumerate}
Then
\begin{eqnarray*}
E' &:=& E \cup \delta^{-1}(\overline{\psi}(E)) \subset \cAp, \\
F' &:=& \rho^{-1}(\iota(\psi(F))) \subset \Fc,
\end{eqnarray*}
are finite sets satisfying the same properties (\ref{p1}),(\ref{p2}),(\ref{p3}) as $E,F$,
and such that
\[ \min \{ \text{length of } c\mid c\in F\} < \min \{ \text{length of } c\mid c\in F'\}. \]
\end{theor}
\begin{proof}
The finiteness of $E',F'$ is obvious.
Notice first that every $d\in \cAp'\backslash\overline{\psi}(E)$ contains an element of $\psi(F)$ since $\overline{\psi}$ is a bijection.
Now let $c\in \cAp\backslash E$.
Then $\delta(c)\in\cAp'$, hence either $\delta(c)\in\overline{\psi}(E)$ or $\delta(c)$ contains an element of $\psi(F)$.
If $\delta(c)\in\overline{\psi}(E)$, then $c\in E'$. Otherwise, $\delta(c)$ strictly contains an element of $\psi(F)$, say $\psi(f)$; but $\psi(f)$ begins with an entry greater than $1$, thus $\delta(c)$ contains $\iota(\psi(f))$. This implies $c\in\rho^{-1}(\iota(\psi(f)))$.
Now every element of $\rho^{-1}(\iota(\psi(f)))$ has length greater than the length of $f$, because $f$ contains a $1$.
\end{proof}

\begin{corol}
Every \quidcyc contains at least one of
\[ ( 0,0),( 1,1), ( 1,2), \text{ or } ( 1,3). \]
\end{corol}
\begin{proof}
If $(1,1)\subseteq c\in\cAp$ then $c=\blp 1,1,1\brp$. Further, any \quidcyc which is not $\blp 0,0\brp$ contains a one.
Thus we may apply Thm.\ \ref{etaseqsubsets} to the sets
\[ E = \{\blp 0,0\brp,\blp 1,1,1\brp\}, \quad F = \{(1)\} \]
and obtain
\begin{eqnarray*}
E' &=& \{\blp 0,0\brp,\blp 1,1,1\brp,\blp 1,2,1,2\brp,\blp 1,2,2,1,3\brp,\blp 1,3,1,3,1,3\brp\}, \\
F' &=& \{(1,2),(2,1),(1,3,1)\},
\end{eqnarray*}
which implies the claim.
\end{proof}

We now formulate the result that we will need to determine Nichols algebras later.

\begin{theor}\label{thm:subseqs}
If $c=(c_1,\ldots,c_n)$ is a representative of a \quidcyc, then either $c$ is one of
$$(0,0),(1,1,1),(1,2,1,2),(2,1,2,1),(2,1,3,1,2)$$
or one of $(c_2,\ldots,c_{n-1})$ or $(c_{n-1},\ldots,c_2)$ contains one of
$$
    ( 1, 2, 2 ),
    ( 1, 2, 3 ),
    ( 1, 2, 4 ),
    ( 2, 1, 3 ),
    ( 2, 1, 4 ), $$
$$  ( 2, 1, 5 ),
    ( 3, 1, 4 ),
    ( 3, 1, 5 ),
    ( 1, 3, 1, 3 ).
$$
\end{theor}
\begin{proof}
Applying Thm.\ \ref{etaseqsubsets} twice we obtain: Every \quidcyc is $\blp 0,0\brp,\blp 1,1,1\brp,\blp 1,2,1,2\brp$ or contains one of
$$  F = \{( 1, 2, 2 ),
    ( 1, 2, 3, 1 ),
    ( 1, 2, 3, 2, 1 ),
    ( 1, 2, 4, 1, 2 ),
    ( 1, 2, 4, 1, 3, 1 ), $$
$$  ( 1, 3, 1, 3 ),
    ( 1, 3, 1, 4, 1 ),
    ( 1, 3, 1, 5, 1, 2 ),
    ( 1, 3, 1, 5, 1, 3, 1 ),$$
$$  ( 1, 4, 1, 2 ),
    ( 2, 1, 3 ),
    ( 2, 1, 5, 1, 2 )\}. $$
Let $c=(c_1,\ldots,c_n)$ be a representative of a \quidcyc and denote $\tilde F$ the union of $F$ and the set of reversed sequences of $F$. We say that $c$ satisfies $(*)$ if
$$c\in E:=\{(0,0),(1,1,1),(1,2,1,2),(2,1,2,1),(2,1,3,1,2)\}$$
or $(c_2,\ldots,c_{n-1})$ contains one of the sequences in $\tilde F$.
For $n<6$, it is easy to check $(*)$ for $c$. Now assume $n\ge 6$.
Choose an ear $i$, i.e.\ $c_i=1$. Then $c':=(c_1,\ldots,c_{i-1}-1,c_{i+1}-1,\ldots,c_n)$ is a representative\footnote{We do not exclude $i\in\{1,n\}$.} of a \quidcyc of length $n-1$. Thus by induction, $c'$ satisfies $(*)$. If $c'\in E$, then it is easy to see that $c$ will satisfy $(*)$ as well.
Thus assume that $(c'_2,\ldots,c'_{n-2})$ contains one of the sequences in $\tilde F$. But including the ear in the sequences of $F$ results in sequences satisfying $(*)$ as well. The claim in the theorem follows from the fact that every sequence in $F$ contains one of the nine sequences above.
\end{proof}

\section{Characteristic sequences and bicharacters}

\subsection{Characteristic sequences}
Let $\qtriple=(q_1,q,q_2)$ be a triple of numbers (in a commutative ring) and assume that
\[ m_i := \min\{m\in\NN_0 \mid 1+q_i+q_i^2+\ldots+q_i^m=0 \text{ or } q_i^m q=1\} \]
for $i=1,2$ are well defined integers. Let
\begin{eqnarray*}
\sigma_1(q_1,q,q_2) &=& (q_1,q_1^{-2m_1}q^{-1},q_1^{m_1^2}q^{m_1}q_2)\\
&=&\begin{cases}(q_1,q_1^2q^{-1},q_1q^{m_1}q_2) & \text{if } 1+q_1+q_1^2+\ldots+q_1^{m_1}=0\\ (q_1,q,q_2) & \text{if } q_1^{m_1} q = 1\end{cases}
\end{eqnarray*}
and similarly
\begin{eqnarray*}
\sigma_2(q_1,q,q_2) &=& (q_1q^{m_2}q_2^{m_2^2},q_2^{-2m_2}q^{-1},q_2).
\end{eqnarray*}
Thus $\sigma_1,\sigma_2$ produce new triples of numbers which possibly define new integers $m_i$, and notice that $\sigma_i(\sigma_i(q_1,q,q_2))=(q_1,q,q_2)$.

\begin{defin}\label{charseq}
Assuming that the new $m_i$ are well defined again and again, the first triple $\qtriple_0:=\qtriple=(q_1,q,q_2)$ will produce an infinite sequence of the form
\[ \ldots \stackrel{\sigma_2}{\longleftrightarrow} \qtriple_{-2} \stackrel{\sigma_1}{\longleftrightarrow} \qtriple_{-1} \stackrel{\sigma_2}{\longleftrightarrow} \qtriple_0 \stackrel{\sigma_1}{\longleftrightarrow} \qtriple_1 \stackrel{\sigma_2}{\longleftrightarrow} \qtriple_2 \stackrel{\sigma_1}{\longleftrightarrow} \ldots \]
where every $\sigma_i$ has its own $m_i$, thus we obtain a sequence of integers
\[ \ldots,c_{-2},c_{-1},c_0,c_1,c_2,\ldots \]
which we call the \emph{characteristic sequence} of $\qtriple=(q_1,q,q_2)$, where
the $c_i$ correspond to the maps in the following way ($c_0=m_1$, $c_{-1}=m_2$):
\[ \ldots \stackrel{c_{-3}}{\longleftrightarrow} \qtriple_{-2} \stackrel{c_{-2}}{\longleftrightarrow} \qtriple_{-1} \stackrel{c_{-1}}{\longleftrightarrow} \qtriple_0 \stackrel{c_0}{\longleftrightarrow} \qtriple_1 \stackrel{c_1}{\longleftrightarrow} \qtriple_2 \stackrel{c_2}{\longleftrightarrow} \ldots \]
We say that a triple $\qtriple$ is \emph{broken} if the above procedure leads to a triple for which one of the $m_i$ is not defined.
\end{defin}

\begin{examp}
Let $\zeta\in\CC$ be a primitive $9$-th root of unity and $\qtriple=(\zeta^6,\zeta^8,\zeta^6)$. Then the above picture is
\[ \ldots
\stackrel{5}{\longleftrightarrow} (\zeta,\zeta^4,\zeta^6) \stackrel{2}{\longleftrightarrow} (\zeta^6,\zeta^8,\zeta^6)
\mathrel{\mathop{\longleftrightarrow}^2_{\sigma_1}}
(\zeta^6,\zeta^4,\zeta) \stackrel{5}{\longleftrightarrow} (\zeta^6,\zeta^4,\zeta) \stackrel{2}{\longleftrightarrow} \ldots \]
and the characteristic sequence is $(\ldots,2,2,5,2,2,5,\ldots)$, thus periodic with period $(2,2,5)$.
\end{examp}

\begin{remar}
The triple $(q_1,q,q_2)$ may be interpreted as an invariant of a \emph{bicharacter} or equivalently of a \emph{Nichols algebra of diagonal type of rank two}. The corresponding characteristic sequence ``is'' its \emph{Weyl groupoid}. In this paper, we classify all these triples such that the corresponding Weyl groupoid is affine, i.e.\ its Tits cone is a half space (see \cite{CMW} for details).
\end{remar}

\subsection{Chains and cycles}

\begin{defin}
If $\sigma_1(\qtriple_i)=\qtriple_i$, then it is convenient to denote the sequence with an \emph{end}:
\[ \ldots \stackrel{\sigma_2}{\longleftrightarrow} \qtriple_{i-2} \stackrel{\sigma_1}{\longleftrightarrow} \qtriple_{i-1} \stackrel{\sigma_2}{\longleftrightarrow} \qtriple_i \stackrel{\sigma_1}{\righttoleftarrow} \]
\end{defin}

\begin{remar}
An easy computation determines all sequences coming from triples in which at least one entry is not a root of unity, all of them have sequences with two ends which we call \emph{chains}.
\end{remar}

\begin{defin}
When all entries in $\qtriple$ are roots of unity, then the possible number of different triples in the sequence for $\qtriple$ is finite. We thus obtain that the characteristic sequence is either a \emph{chain} or a \emph{cycle}:
\begin{eqnarray*}
\text{chain:} && \lefttorightarrow \qtriple_j \longleftrightarrow \qtriple_{j+1} \longleftrightarrow \ldots \longleftrightarrow \qtriple_{i-1} \longleftrightarrow \qtriple_i \righttoleftarrow \\
\text{cycle:} && \qtriple_i = \qtriple_j \longleftrightarrow \qtriple_{j+1} \longleftrightarrow \ldots \longleftrightarrow \qtriple_{i-1} \longleftrightarrow \qtriple_i
\end{eqnarray*}
\end{defin}

\subsection{Bicharacters from characteristic sequences}

To determine the triple $\qtriple$ from a given characteristic sequence, the knowledge of three consecutive entries $c_i,c_{i+1},c_{i+2}$ is sufficient.
However, this is only possible if $i+1$ is not an end of the sequence, i.e.\ the sequence is not a chain ending at position $i+1$, since in this case $c_i=c_{i+2}$ and we actually only ``know'' two entries.
\begin{propo}
We may completely reconstruct the sequence of triples from a characteristic sequence provided we know three entries $c_i$ as in one of the following situations,
\begin{eqnarray}
& \ldots \qtriple_{i} \stackrel{c_i}{\longleftrightarrow} \qtriple_{i+1} \stackrel{c_{i+1}}{\longleftrightarrow} \qtriple_{i+2} \stackrel{c_{i+2}}{\longleftrightarrow} \qtriple_{i+3} \ldots \\
\label{sit2}& \ldots \qtriple_{i-2} \stackrel{c_{i-2}}{\longleftrightarrow} \qtriple_{i-1} \stackrel{c_{i-1}}{\longleftrightarrow} \qtriple_i \stackrel{c_i}{\righttoleftarrow} \\
& \stackrel{c_0}{\lefttorightarrow} \qtriple_{1} \stackrel{c_{1}}{\longleftrightarrow} \qtriple_2 \stackrel{c_2}{\righttoleftarrow} \end{eqnarray}
where we have included all possible ends.
\end{propo}
\begin{proof}
Every situation above produces a system of equations in $q_1,q,q_2$ such that all solutions produce the same characteristic sequence.
\end{proof}

\begin{examp}
If we know that the characteristic sequence contains $\ldots,1,3,1,\ldots$, then this information is not sufficient since $3$ could be an end (as in situation (\ref{sit2})).
\end{examp}

\section{Affine Weyl groupoids of rank two}

One obtains all \quidcycs by starting with $(0,0)$ and repeatedly including $1$'s as in Def.\ \ref{eta_seq}.
Similarly, one obtains characteristic sequences of affine Weyl groupoids by starting with the infinite sequence $(\ldots,2,2,\ldots)$ (affine type $A_1^{(1)}$) and repeatedly including $1$'s, however, possibly infinitely many.

Since the sequences we are interested in are periodic (when $(q_1,q,q_2)$ are roots of unity), this amounts to the following construction \cite{CMW}:

\begin{theor}\label{affwgd:ranktwo}
Let $\mathcal{W}$
be an affine Weyl groupoid of rank two with limit $A_1$ with characteristic sequence $\cc$.
Then there exist representatives of \quidcycs $\cc'_k=(c'_{k,1},\ldots,c'_{k,r_k})$, $r_k\in\NN_{>1}$, $k\in\ZZ$ such that
\[
\begin{matrix}
\cc = (\ldots,c'_{k-1,r_{k-2}},&c'_{k-1,r_{k-1}}\\
&+2+ \\
&c'_{k,1}, & c'_{k,2},\ldots,c'_{k,r_k-1}, & c'_{k,r_k} \\
&&& +2+\\
&&&c'_{k+1,1}, & c'_{k+1,1},\ldots).
\end{matrix}
\]
\end{theor}

\begin{defin}
We call a sequence of the structure given in Thm.\ \ref{affwgd:ranktwo} \emph{affine}.
A triple $\qtriple=(q_1,q,q_2)$ is called \emph{affine} if its characteristic sequence is affine.
We denote by $\mu_n$ the set of primitive $n$-th roots of unity.
\end{defin}

\begin{corol}\label{15cases}
Every affine sequence $\cc$ contains one of
\begin{eqnarray*}
(*)&&
    ( 1, 2, 2 ),
    ( 1, 2, 3 ),
    ( 1, 2, 4 ),
    ( 2, 1, 3 ),
    ( 2, 1, 4 ),\\
&&  ( 2, 1, 5 ),
    ( 3, 1, 4 ),
    ( 3, 1, 5 ),
    ( 1, 3, 1, 3 ),
\\
(**)&&
    ( 1, 3, 2 ),
    ( 1, 3, 3 ),
    ( 1, 4, 1, 4 ),
    ( 2, 1, 6 ),
    ( 2, 2, 2, 2 ),
    ( 3, 1, 6 ).
\end{eqnarray*}
\end{corol}
\begin{proof}
We assume that $\cc$ is made of representatives of \quidcycs $\cc'_k$. If one of them has length at least $6$, then by Thm.\ \ref{thm:subseqs}, it contains one of the sequences in $(*)$ in the middle, hence $\cc$ contains one of those.
Otherwise, all the $\cc'_k$ have length at most $5$. If one of them has length $5$ and is not $(2,1,3,1,2)$, then it contains one of $(2,1,3)$, $(3,1,2)$, $(1,2,2)$, or $(2,2,1)$ in the middle, which is include in $(*)$. We are now left with a sequence $\cc$ composed of $\cc'_k$ out of $E:=\{(0,0),(1,1,1),(1,2,1,2),(2,1,2,1),(2,1,3,1,2)\}$. But putting triples of elements of $E$ together as in Thm.\ \ref{affwgd:ranktwo} always produces sequences containing at least one of the subsequences of $(*)$ or $(**)$.
\end{proof}

To construct all affine triples $\qtriple$ it is thus sufficient to consider the $15$ cases presented in Cor.\ \ref{15cases}:

\begin{theor}
If $\qtriple=(q_1,q,q_2)$ is an affine triple, then up to permutations of labels, its \emph{Dynkin diagram}
$$\begin{tikzpicture}
\draw (0,0) -- (1.2,0);
\draw (-0.08,0) circle (0.08);
\draw (1.28,0) circle (0.08);
\coordinate[label=0:$q_1$] (q1) at (-0.32,0.36);
\coordinate[label=0:$q$] (q) at (0.4,0.30);
\coordinate[label=0:$q_2$] (q2) at (1.0,0.36);
\end{tikzpicture}$$
is one of those listed in Fig.\ \ref{dynkaff}. The period of its corresponding characteristic sequence is given in the last column.
\end{theor}

\begin{figure}
\begin{tabular}{r|l|l|l}
& generalized Dynkin diagram & parameters & period \\
\hline
\hline
1 & \dynk{\zeta}{\zeta}{\zeta} & $\zeta\in\mu_3$ & $(2)$ \\
\hline
2 & \dynk{\zeta^2}{\zeta^5}{\zeta^2} & $\zeta\in\mu_6$ & $(2)$ \\
\hline
3 & \dynk{\zeta}{\zeta^4}{\zeta^4} & $\zeta\in\mu_6$ & $(2)$ \\
\hline
4 & \dynk{\zeta^4}{\zeta}{\zeta^2} \dynk{\zeta^2}{\zeta^3}{\zeta^2} \dynk{\zeta^2}{\zeta}{\zeta^4} & $\zeta\in\mu_6$ & $(2)$ \\
\hline
5 & \dynk{\zeta}{\zeta^{10}}{\zeta^4} & $\zeta\in\mu_{12}$ & $(2)$ \\
\hline
6 & \dynk{\zeta}{\zeta^4}{\zeta^4} & $\zeta\in\mu_5$ & $(1,4)$ \\
\hline
7 & \dynk{\zeta^4}{\zeta^4}{\zeta} & $\zeta\in\mu_8$ & $(1,4)$ \\
\hline
8 & \dynk{\zeta}{\zeta^9}{\zeta^4} & $\zeta\in\mu_{10}$ & $(1,4)$ \\
\hline
9 & \dynk{\zeta}{\zeta^{10}}{\zeta^9} \dynk{\zeta^9}{\zeta^8}{\zeta^4} & $\zeta\in\mu_{12}$ & $(2,3,1,3)$ \\
\hline
10 &
\begin{minipage}{160pt}
\dynk{\zeta}{\zeta^8}{\zeta^6}
\dynk{\zeta^6}{\zeta^4}{\zeta^3}
\dynk{\zeta^3}{\zeta^2}{\zeta^9}
\dynk{\zeta^9}{\zeta^4}{\zeta^6}
\dynk{\zeta^6}{\zeta^8}{\zeta^7}
\end{minipage}
& $\zeta\in\mu_{12}$ & $(4,1,3,3,1)$ \\
\hline
11 & \dynk{\zeta}{\zeta^{12}}{\zeta^9}
\dynk{\zeta^9}{\zeta^6}{\zeta^4} & $\zeta\in\mu_{18}$ & $(6,1,3,1)$ \\
\hline
12 & \dynk{q}{q^{-2}}{q} & $q \in\CC^\times\backslash\{\pm 1\} $ & $(2)$ \\
\hline
13 & \dynk{q}{q^{-2}}{-q} & $q \in\CC^\times\backslash\{\pm 1\} $ & $(2)$ \\
\hline
14 & \dynk{q}{q^{-4}}{q^4} &
\begin{minipage}{70pt}
$q \in\CC^\times\backslash\{\pm 1\}$,
$q \notin \mu_3$, $q \notin \mu_4$
\end{minipage}
& $(1,4)$ \\
\end{tabular}
\caption{Generalized Dynkin diagrams of affine bicharacters of rank two.\label{dynkaff}}
\end{figure}

\noindent
{\bf Acknowledgment.} Many thanks to H.\ Yamane for pointing out that line 13 in Fig.\ \ref{dynkaff} was missing in a previous version.

% \newpage

\bibliographystyle{amsplain}
% \bibliography{../references/refs}
\def\cprime{$'$}
\providecommand{\bysame}{\leavevmode\hbox to3em{\hrulefill}\thinspace}
\providecommand{\MR}{\relax\ifhmode\unskip\space\fi MR }
\providecommand{\MRhref}[2]{%
  \href{http://www.ams.org/mathscinet-getitem?mr=#1}{#2}
}
\providecommand{\href}[2]{#2}

\end{document}